\documentclass[11pt]{amsart}

\usepackage{amsmath}

\usepackage{amsfonts}

\usepackage{amssymb}

\numberwithin{equation}{section}

\newtheorem{lem}{Lemma}[section]

\newtheorem{teo}[lem]{Theorem}

\newtheorem{oss}[lem]{Remark}

\newcommand{\II}{I\!I}

\renewcommand{\phi}{\varphi}

\newcommand{\R}{{\mathbb{R}}}

\newcommand{\bbR}{{\mathbb{R}}}

\newcommand{\Ric}{{\mathrm{Ric}}}

\newcommand{\average}{{\mathchoice {\kern1ex\vcenter{\hrule height.4pt

width 6pt

depth0pt} \kern-9.7pt} {\kern1ex\vcenter{\hrule height.4pt width 4.3pt

depth0pt}

\kern-7pt} {} {} }}

\title[Classification of stable solutions]{Classification of stable solutions\\
for boundary value
problems\\ with nonlinear boundary conditions\\ on Riemannian manifolds\\
with nonnegative Ricci curvature}

\thanks{Part of this work was done while A. P. was visiting the
Dipartimento di Matematica ``Federigo Enriques"
of the University of Milan. 
The authors are members of {\em Gruppo Nazionale per l'Analisi
Ma\-te\-ma\-ti\-ca, la Probabilit\`a e le loro Applicazioni} (GNAMPA)
of the {\em Istituto Nazionale di Alta Matematica} (INdAM). Supported by
the Australian Research Council Discovery Project Grant ``N.E.W. Nonlocal Equations
at Work''.}

\author{Serena Dipierro, Andrea Pinamonti, Enrico Valdinoci}

\address{Dipartimento di Matematica ``Federigo Enriques'',
Universit\`a degli studi di Milano,
Via Saldini 50, 20133 Milano (Italy).}

\email{serena.dipierro@unimi.it}

\address{Dipartimento di Matematica, Universit\`a di Trento,
Via Sommarive 14, 38050 Povo, Trento (Italy).}

\email{andrea.pinamonti@gmail.com}

\address{Dipartimento di Matematica ``Federigo Enriques'',
Universit\`a degli studi di Milano,
Via Saldini 50, 20133 Milano (Italy), and
School of Mathematics and Statistics,
University of Melbourne, Grattan Street, 
Parkville, VIC-3010 Melbourne (Australia), and
Istituto di Matematica Applicata e Tecnologie Informatiche,
Via Ferrata 1, 27100 Pavia (Italy).
}

\email{enrico.valdinoci@unimi.it}

\begin{document}

\begin{abstract}
We present
a geometric formula of Poincar\'e type,
which is
inspired by a classical work of Sternberg and
Zumbrun, and
we provide a classification result
of stable solutions of linear elliptic problems
with nonlinear Robin conditions on Riemannian manifolds
with nonnegative Ricci curvature.

The result obtained here is a refinement of
a result recently established by Bandle, Mastrolia, 
Monticelli and Punzo.
\end{abstract}

\maketitle

%\tableofcontents

\section{Introduction}  

The study of partial differential equations on manifolds
has a long tradition in analysis and geometry, see e.g.~\cite{MORR, W71, W73, Jost, A98}.
The interest for such topic may come from different perspectives:
on the one hand, at a local level, classical equations with variable coefficients
can be efficiently comprised into the manifold setting,
allowing more general and elegant treatments;
in addition, at a global level,
the geometry of the manifold can produce new interesting phenomena
and interplay with the structure of the solutions thus creating a
novel scenario for the problems into consideration.

Of course, given the complexity of the topic, the different solutions
of a given partial differential equation on a manifold can give rise to
a rather wild ``zoology'' and it is important to try to group the solutions
into suitable ``classes'' and possibly to classify all the solution belonging to a class.

In this spirit, very natural classes of solutions in a variational setting
arise from energy considerations. The simplest class in this framework
is probably that of ``minimal solutions'', namely the class of solutions
which minimize (or, more generally, local minimize) the energy functional.

On the other hand, it is often useful to look at a more general class
than minimal solutions, that is the class of
solutions at which the second derivative of the energy functional
is nonnegative.
These solutions are called ``stable'' (see e.g.~\cite{DUPAI}).
Of course, the class of stable solutions contains that of minimal solutions,
but the notion of stability is often in concrete situations
more treatable than that of minimality: for instance,
it is typically very difficult to establish whether 
or not a given solution is minimal, since one in principle should
compare its energy with that of all the possible competitors,
while a stability check could be more manageable, relying
only on a single, and sometimes
sufficiently explicit, second derivative bound.\medskip

The goal of this paper is to study the case of a linear elliptic equation
on a domain of a Riemannian manifold
with nonnegative Ricci curvature, endowed with nonlinear
boundary data. We will consider stable solutions
in this setting and provide sufficient conditions
to ensure that they are necessarily constant.\medskip

The framework in which we work is the following.
Let~$M$ be a connected~$m$-dimensional Riemannian manifold
endowed with a smooth Riemannian metric~$g=(g_{ij})$.
We denote by~$\Delta$ the Laplace-Beltrami operator
induced by~$g$.
Let~$\Omega\subset M$ be a compact orientable domain and~~$\nu$ be
the outer normal vector of~$\partial\Omega$ lying in the tangent
space~$T_p M$ for any~$p\in\partial\Omega$.
We assume that~$\partial\Omega$ is orientable for the outer normal
to be well defined and continuous.

In this paper we study the solutions to the following boundary value problem:
\begin{align}\label{eqncomp}
\bigg \{
\begin{array}{rl}
\Delta u+f(u)=0 & \mbox{in}\ \Omega, \\
\partial_{\nu}u+h(u)=0 & \mbox{on}\ \partial\Omega, \\
\end{array}
\end{align} 
where~$f,h \in C^{1}(\bbR)$ and~$\partial_{\nu}u:= g(\nabla u, \nu)$. Similar problems
have been investigated in~\cite{DPV, punzo, punzo1, Jimbo}.

As usual, we consider the volume term induced by~$g$,
that is, in local coordinates,
\begin{equation*}
dV = \sqrt{|g|}\,dx^1\wedge \dots \wedge dx^m,
\end{equation*}
where~$\{ dx^1,\dots,dx^m\}$ is the basis of~$1$-forms
dual to the vector basis~$\{\partial_1,\dots,\partial_m\}$,
and~$|g|=\det(g_{ij})\ge0$. We also denote by~$d\sigma$ the volume measure on~$\partial\Omega$ induced by the embedding~$\partial\Omega\hookrightarrow M$.\medskip

As customary, we say that~$u$ is a weak solution to~\eqref{eqncomp}
if~$u\in C^1(\overline{\Omega})$ and
\begin{align}\label{weak}
\int_{\Omega}\left\langle \nabla u,
\nabla\varphi\right\rangle \, dV
+\int_{\partial\Omega} h(u)\varphi \,
d \sigma=\int_{\Omega} f(u)\varphi \, dV,
\quad {\mbox{ for any }} \varphi\in C^1(\Omega).
\end{align}
Moreover,
we say that a weak solution~$u$ is stable if
\begin{align}\label{hstab}
\int_{\Omega}|\nabla\varphi|^2 \, dV+\int_{\partial\Omega}
h'(u)\varphi^2 \, d\sigma-
\int_{\Omega} f'(u)\varphi^2\, dV\ge 0, \qquad {\mbox{ for any }}
\varphi\in C^{1}(\Omega).
\end{align} 
\medskip

In order to state our result we recall below some classical notions in Riemannian geometry.
Given a vector field~$X$, we denote
$$ |X|=\sqrt{\langle X,X\rangle}.$$
Also (see, for instance Definition~3.3.5 in~\cite{Jost}),
it is customary to define
the Hessian of a smooth function~$\phi$ as
the symmetric~$2$-tensor given in a local patch by
$$ (H_\phi)_{ij}=\partial^2_{ij}\phi-\Gamma^k_{ij}\partial_k\phi,$$
where~$\Gamma^k_{ij}$ are the  
Christoffel symbols, namely
$$ 
\Gamma_{ij}^k=\frac12 g^{hk} \left( \partial_i g_{hj} +\partial_j g_{ih} -\partial_h g_{ij} \right) .$$
Given a tensor~$A$,
we define its norm by~$|A|=\sqrt{A A^*}$, where~$A^*$
is the adjoint. 

The above quantities are related to the Ricci tensor~$\Ric$
via the Bochner-Weitzenb\"ock formula (see,
for instance,~\cite{Berger} and references therein):
\begin{equation}\label{BOC}
\frac 12\Delta |\nabla \phi|^2=
|H_\phi|^2+ \langle \nabla \Delta \phi,\nabla\phi\rangle
+\Ric (\nabla \phi,\nabla\phi).\end{equation}
Finally, we let~$\II$ and~$H$ denote the second fundamental
tensor and the mean curvature of the
embedding~$\partial\Omega \hookrightarrow \Omega$ in the
direction of the outward unit normal vector field~$\nu$, respectively.
\medskip

We are now in position to state our main result:
\begin{teo}\label{main1}
Let~$u\in C^3(\overline{\Omega})$ be a stable solution
to~\eqref{eqncomp}. Assume that the Ricci curvature is nonnegative
in~$\Omega$, and that, for any~$p\in\partial\Omega$,
\begin{equation}\label{cond0}\begin{split}&
{\mbox{the quadratic form~$\II-h'(u)\, \tilde g$ on the tangent
space~$T_p(\partial\Omega)$}}\\&{\mbox{is nonpositive definite.}}\end{split}\end{equation}
If
\begin{equation}\label{cond}
\int_{\partial\Omega}\Big(h(u)\, f(u)+ (m-1)\,\big(h(u)\big)^2\, H
+h'(u)\big(h(u)\big)^2\Big)\, d\sigma \le 0,
\end{equation}
then~$u$ is constant in~$\Omega$.
\end{teo}

\begin{oss}{\rm
Theorem~\ref{main1} has been proved
in~\cite[Theorem 4.5]{punzo} in the particular case
in which~$h(t):=\alpha t$,
for some~$\alpha\in\R$. We point out that,
with this particular choice
of~$h$, Theorem~\ref{main1} here
weakens the assumptions on 
the sign of~$\alpha$ of~\cite[Theorem 4.5]{punzo}.}
\end{oss}

The proof of Theorem~\ref{main1} is
based on a geometric Poincar\'e-type inequality, which we state
in this setting as follows:

\begin{teo}\label{eqnlin}
Let~$u$ be stable weak solution to~\eqref{eqncomp}. Then,
\begin{equation}\begin{split}\label{GF}
&\int_\Omega \Big(
\Ric (\nabla u,\nabla u)+|H_u|^2-\big|\nabla|\nabla u|\big|^2
\Big)\varphi^2\,dV\\&\qquad
-\int_{\partial\Omega} \left( \frac{1}{2}\left\langle \nabla|\nabla u|^2, 
\nu\right\rangle
+h'(u) |\nabla u|^2 \right)\varphi^2 \, d\sigma\\
\le &
\int_\Omega |\nabla u|^2 |\nabla\varphi|^2\,dV,\end{split}\end{equation}
for any~$\varphi\in C^\infty(\Omega)$.
\end{teo}

We notice that formula~\eqref{GF}
relates
the stability condition of the solution with the
principal curvatures and the tangential gradient
of the corresponding
level set. Since this formula
bounds a weighted~$L^2$-norm of
any~$\varphi\in C^1(\Omega)$ plus a
boundary term by a weighted~$L^2$-norm of its gradient,
we may consider this
formula as a weighted Poincar\'e type inequality.
\medskip

The idea of
using weighted Poincar\'e inequalities to deduce quantitative
and qualitative information on the solutions of a partial differential
equation has been originally introduced by Sternberg and
Zumbrun in~\cite{SZ1, SZ2}
in the context of the Allen-Cahn equation, and it has
been extensively exploited to prove symmetry
and rigidity results, see e.g.~\cite{FarHab, FSV, ASV2}.
See also~\cite{FMV, fsv1, fsv2, FP, FV1, PV}
for applications to Riemannian and sub-Riemannian manifolds,
\cite{CNV} for problems involving the Ornstein-Uhlenbeck operator,
\cite{CNP,FNP}
for semilinear equations with unbounded drift and~\cite{fazly, DI, DP1,DP2}
for systems of equations. \medskip

Recently, in~\cite{DSV, DPV},
the cases of Neumann conditions for boundary
reaction-diffusion equations
and of Robin conditions for linear and quasilinear equations
have been studied,
using a Poincar\'e inequality that involves also suitable boundary terms. 
\medskip

We point out that Theorem~\ref{main1} comprises the classical case
of the Laplacian in the Euclidean space with homogeneous
Neumann data, which was studied in
the celebrated papers~\cite{CH, Mat}. In this spirit,
our Theorem~\ref{main1} can be seen as a nonlinear
version of the results of~\cite{CH, Mat} on Riemannian manifolds
(and, with respect to~\cite{CH, Mat},
we perform a technically different proof, based on
Theorem~\ref{eqnlin}).
\medskip

The next two sections are devoted to the proofs of
Theorems~\ref{eqnlin} and~\ref{main1} respectively.

\section{Proof of Theorem~\ref{eqnlin}}

Applying~\eqref{hstab} with~$\varphi$
replaced by~$|\nabla_g u| \varphi$, we get
\begin{eqnarray*}
&& \int_\Omega f'(u)|\nabla u|^2 \phi^2\,dV\\
&\le& 
\int_\Omega \Big(\big|\nabla |\nabla u|\big|^2 \phi^2
+|\nabla u|^2|\nabla \phi|^2 +2\phi|\nabla u|
\langle\nabla \phi,\nabla|\nabla u|\rangle\Big)\,dV\\
&&\qquad +\int_{\partial\Omega} h'(u) |\nabla u|^2 \varphi^2 \,d\sigma\\
&=&
\int_\Omega\Big( \big|\nabla |\nabla u|\big|^2 \phi^2
+|\nabla u|^2|\nabla \phi|^2 +\frac12
\langle\nabla \phi^2,\nabla |\nabla u|^2\rangle\Big)\,dV+\int_{\partial\Omega} h'(u) |\nabla u|^2 \varphi^2 d\sigma
.\end{eqnarray*}
Therefore, integrating by parts the third term in the last line, we get
\begin{eqnarray*}
&&\int_\Omega f'(u)|\nabla u|^2 \phi^2\,dV\\
&\le&
\int_\Omega \Big(\big|\nabla |\nabla u|\big|^2 \phi^2
+|\nabla u|^2|\nabla \phi|^2 -\frac12
\phi^2\Delta |\nabla u|^2\Big)\,dV\\
&&\qquad +\frac{1}{2}\int_{\partial\Omega} \varphi^2 \left\langle
\nabla|\nabla u|^2, \nu\right\rangle\, d\sigma
+\int_{\partial\Omega} h'(u) |\nabla u|^2 \varphi^2 \,d\sigma.
\end{eqnarray*}
Hence, recalling~\eqref{BOC},
\begin{equation}\begin{split}\label{ieegherhrhg}
&\int_\Omega f'(u)|\nabla u|^2 \phi^2\,dV\\
\le\;&
\int_\Omega \Big[\big|\nabla |\nabla u|\big|^2 \phi^2
+|\nabla u|^2|\nabla \phi|^2
-\Big(
|H_u|^2+ \langle \nabla \Delta u,\nabla u\rangle
+\Ric (\nabla u,\nabla u)
\Big)\varphi^2\Big]
\,dV\\
&\qquad+\frac{1}{2}\int_{\partial\Omega} \varphi^2 \left\langle \nabla|\nabla u|^2,
\nu\right\rangle \, d\sigma
+\int_{\partial\Omega} h'(u) |\nabla u|^2 \varphi^2\, d\sigma
.\end{split}\end{equation}
Now, by differentiating the equation in~\eqref{eqncomp}, we see that
$$ -\nabla \Delta u=f'(u)\nabla u.$$
Plugging this information into~\eqref{ieegherhrhg}, we conclude that
\begin{eqnarray*}
0&\le&
\int_\Omega \Big[ \big|\nabla |\nabla u|\big|^2 \phi^2
+|\nabla u|^2|\nabla \phi|^2 -
\Big(
|H_u|^2+\Ric (\nabla u,\nabla u)
\Big)\phi^2\Big]\,dV+\\
&&\qquad 
+\frac{1}{2}\int_{\partial\Omega} \varphi^2 \left\langle \nabla|\nabla u|^2, \nu\right\rangle d\sigma+\int_{\partial\Omega} h'(u) |\nabla u|^2 \varphi^2 d\sigma
,\end{eqnarray*}
which completes the proof of Theorem~\ref{eqnlin}.\qed

\section{Proof of Theorem~\ref{main1}}

In this section we provide the proof of Theorem~\ref{main1}. 
We first state 
the following result, that proves Theorem~3.4
of~\cite{punzo} in the more general case in which~$h$ is any~$C^1$
function.

\begin{teo}\label{Pz}
Let~$w\in C^3(\overline{\Omega})$ satisfy
\begin{equation}\label{bejgbegje}
\partial_{\nu} w+h(w)=0\quad {\mbox{ on }}
\partial \Omega,
\end{equation}
for some~$h\in C^1(\R)$. Then
\[
\frac{1}{2}\frac{\partial}{\partial\nu} |\nabla w|^2
= \II(\tilde{\nabla} w,\tilde{\nabla} w)
-h'(w)|\tilde{\nabla} w|^2-h(w)\, H_w(\nu,\nu)
\quad {\mbox{ on }} \partial\Omega,
\]
where~$\tilde{\nabla} w:=\nabla w - g(\nabla w,\nu)\nu$
is the tangential gradient with respect to~$\partial\Omega$,
and~$H_w$ is the Hessian matrix of the function~$w$.
\end{teo}

\begin{proof} 
We let~$\{e_i\}$, with~$i\in\{1,\dots,m\}$,
be a Darboux frame along~$\partial\Omega$, that is such that~$
e_m:=\nu$. In this setting, conditions~\eqref{bejgbegje} reads
\begin{equation}\label{bejgbegje:1}
w_m=-h(w)\quad {\mbox{ on }}
\partial \Omega.
\end{equation}
Also, for any~$i$,~$j\in\{1,\dots, m-1\}$, we define
$$ H_{ij}:=g\big(\II(e_i,e_j),\nu\big).$$
Then, reasoning as in the proof of formula~(3.32) 
in~\cite{punzo}, we obtain that, for any~$j\in\{1,\dots,m-1\}$,
$$ w_{jm}= \sum_{i=1}^{m-1}H_{ij}w_i -h'(w)w_j\quad
{\mbox{ on }} \partial\Omega.~$$
Therefore, multiplying both terms by~$w_j$, we get
\begin{equation}\label{trruhgrh}
w_{jm}\,w_j=\sum_{i=1}^{m-1} H_{ij}w_i\,w_j -h'(w) w_j^2\quad
{\mbox{ on }} \partial\Omega.
\end{equation}
On the other hand, for any~$i\in\{1,\dots,m\}$,
$$ \frac12 \big( |\nabla w|^2\big)_i= 
\sum_{j=1}^m w_j\,w_{ji}=
\sum_{i=1}^{m-1}w_j\,w_{ji}+ w_m\,w_{mi} =
\sum_{i=1}^{m-1}w_j\,w_{ji}-h(w)w_{mi},
$$
where we used~\eqref{bejgbegje:1} in the last passage.

As a consequence, 
$$ \frac12 \frac{\partial}{\partial\nu} |\nabla w|^2= 
\sum_{j=1}^{m-1}w_j\,w_{jm}-h(w)w_{mm}\quad {\mbox{ on }}
\partial\Omega.
$$
{F}rom this and~\eqref{trruhgrh} we thus obtain 
$$ \frac12 \frac{\partial}{\partial\nu} |\nabla w|^2= 
\sum_{i,j=1}^{m-1}H_{ij}w_i\,w_j -
h'(w)\,\sum_{j=1}^{m-1}w_j^2-h(w)w_{mm}\quad
{\mbox{ on }}
\partial\Omega,
$$
which implies the desired result.
\end{proof}

Now we recall that~$\Delta$ is the Laplace-Beltrami
operator of the manifold~$(M,g)$, and we 
let~$\tilde{\Delta}$ be the Laplace-Beltrami operator of the
manifold~$\partial\Omega$ endowed with the induced metric
by the embedding~$\partial\Omega\hookrightarrow M$.
It holds that
\begin{align}\label{sss}
\Delta w= \tilde{\Delta} w -(m-1)\,H\,\frac{\partial w }{\partial \nu}+ H_w(\nu,\nu).
\end{align}
With this, we can prove the following result:

\begin{lem}
Let~$u\in C^3(\overline{\Omega})$ be a stable solution of~\eqref{eqncomp}. Then
\begin{equation}\begin{split}\label{GF3}
&\int_\Omega \Big(
\Ric (\nabla u,\nabla u)+|H_u|^2-\big|\nabla|\nabla u|\big|^2
\Big)\varphi^2\,dV\\
&\qquad
-\int_{\partial\Omega} \left(\II(\tilde{\nabla} u,\tilde{\nabla} u)-
h'(u) |\tilde{\nabla} u|^2+h(u)\, f(u)+
(m-1)\big(h(u)\big)^2 H
+h'(u)\big(h(u)\big)^2  \right) \varphi^2 \, d\sigma\\
\le\;&
\int_\Omega |\nabla u|^2 |\nabla\varphi|^2
\,dV-\int_{\partial\Omega}h( u)\,\left\langle
\tilde{\nabla} u,\tilde{\nabla}\varphi^2\right\rangle\, d\sigma,
\end{split}\end{equation}
for any~$\varphi\in C^\infty(\Omega)$.
\end{lem}

\begin{proof}
{F}rom Theorems~\ref{eqnlin} and~\ref{Pz}, 
for every stable weak solution~$u$ to~\eqref{eqncomp}
and for any~$\varphi\in C^\infty(\Omega)$, we have that
\begin{equation}\begin{split}\label{jgjbr9tuy}
&\int_\Omega \Big(
\Ric (\nabla u,\nabla u)+|H_u|^2-\big|\nabla|\nabla u|\big|^2
\Big)\varphi^2\,dV\\
&\qquad - \int_{\partial\Omega}\Big(
\II(\tilde{\nabla} u,\tilde{\nabla} u)-h'(u) |\tilde{\nabla} u|^2
-h(u)\,
H_u(\nu,\nu)+h'(u)|\nabla u|^2\Big)\varphi^2\,d\sigma\\
\le\;& 
\int_\Omega |\nabla u|^2 |\nabla\varphi|^2\,dV.
\end{split}\end{equation}
Now we use~\eqref{sss} to manipulate the integral on the boundary
of~$\Omega$: in this way, we obtain from~\eqref{jgjbr9tuy} that
\begin{equation*}\begin{split}%\label{GF2}
&\int_\Omega \Big(
\Ric (\nabla u,\nabla u)+|H_u|^2-\big|\nabla|\nabla u|\big|^2
\Big)\varphi^2\,dV\\
&\qquad- \int_{\partial\Omega}\left[
\II(\tilde{\nabla} u,\tilde{\nabla} u)-h'(u) |\tilde{\nabla} u|^2
-h(u)
\left(\Delta u -\tilde{\Delta} u +(m-1)\,H\,\frac{\partial u }{\partial \nu}\right)
+h'(u)|\nabla u|^2\right]\varphi^2\,d\sigma\\
\le\;&
\int_\Omega |\nabla u|^2 |\nabla\varphi|^2\,dV.\end{split}\end{equation*}
Thus, recalling\footnote{Notice that,
since~$u$ is regular enough,
the equation holds true up to the boundary of~$\Omega$.}
\eqref{eqncomp}, we conclude that
\begin{equation}\begin{split}\label{GF2}
&\int_\Omega \Big(
\Ric (\nabla u,\nabla u)+|H_u|^2-\big|\nabla|\nabla u|\big|^2
\Big)\varphi^2\,dV\\
&\qquad- \int_{\partial\Omega}\Big(
\II(\tilde{\nabla} u,\tilde{\nabla} u)-
h'(u) |\tilde{\nabla} u|^2+h(u)\,
f( u)+h(u)\,\tilde{\Delta} u +
(m-1)\,\big(h(u)\big)^2\,H
+h'(u)|\nabla u|^2\Big)\varphi^2\,d\sigma\\
\le\;&
\int_\Omega |\nabla u|^2 |\nabla\varphi|^2\,dV.\end{split}\end{equation}
Now we observe that 
\[
|\nabla u|^2=|\tilde{\nabla} u|^2+\left|\frac{\partial u}{\partial\nu}\right|^2
=|\tilde{\nabla} u|^2+\big(h(u) \big)^2\quad {\mbox{ on }} \partial\Omega.
\]
Plugging this information into~\eqref{GF2}, we obtain that
\begin{equation*}\begin{split}
&\int_\Omega \Big(
\Ric (\nabla u,\nabla u)+|H_u|^2-\big|\nabla|\nabla u|\big|^2
\Big)\varphi^2\,dV\\
&\qquad - \int_{\partial\Omega}\Big(
\II(\tilde{\nabla} u,\tilde{\nabla} u)
+h(u)\,
f( u)+h(u)\,\tilde{\Delta} u +
(m-1)\,\big(h(u)\big)^2\,H +h'(u)\big(h(u)\big)^2 \Big)\varphi^2\,d\sigma\\
\le\;&
\int_\Omega |\nabla u|^2 |\nabla\varphi|^2\,dV.\end{split}\end{equation*}
Now we notice that
\[
\int_{\partial\Omega} h(u)\,\varphi^2\,\tilde{\Delta} u \, d\sigma=
-\int_{\partial\Omega}h'(u) |\tilde{\nabla}u|^2\varphi^2\, d\sigma
-\int_{\partial\Omega} h(u)\left\langle \tilde{\nabla} u,\tilde{\nabla}\varphi^2\right\rangle\, d\sigma,
\]
and therefore
\begin{eqnarray*}
&&\int_\Omega \Big(
\Ric (\nabla u,\nabla u)+|H_u|^2-\big|\nabla|\nabla u|\big|^2
\Big)\varphi^2\,dV\\
&&\qquad -\int_{\partial\Omega}
\left(\II(\tilde{\nabla} u,\tilde{\nabla} u)-h'(u) |\tilde{\nabla} u|^2
+h(u) \,f(u)+
(m-1)\,\big(h(u)\big)^2 H+
h'(u) \big(h(u)\big)^2  \right) \varphi^2 \,d\sigma \\
&\le&
\int_\Omega |\nabla u|^2 |\nabla\varphi|^2
\,dV-\int_{\partial\Omega} h(u)\,\left\langle
\tilde{\nabla} u,\tilde{\nabla}\varphi^2\right\rangle\, d\sigma,
\end{eqnarray*}
which proves the desired inequality.
\end{proof}

Before completing the proof of
Theorem~\ref{main1} we recall the following lemmata
proved in~\cite[Lemma 5]{fsv1} and~\cite[Lemma 9]{fsv1}, respectively.

\begin{lem}\label{stimas}
For any smooth function~$\phi:M\rightarrow\R$, we have that
\begin{equation}\label{POS}
|H_\phi|^2\ge\big|\nabla|\nabla \phi|\big|^2\qquad
{\mbox{ almost everywhere.}}
\end{equation}
\end{lem}

\begin{lem} \label{Fi1}
Suppose that the Ricci curvature of~$M$
is nonnegative and that~$\Ric$
does not vanish identically.

Let~$u$ be a solution
of~\eqref{eqncomp}, with
\begin{equation*}
\Ric (\nabla u,\nabla u) (p)=0\qquad {\mbox{ for any }}
p\in M.\end{equation*}
Then,~$u$ is constant.
\end{lem}

With this, we are able to finish the proof of Theorem~\ref{main1}:

\begin{proof}[Proof of Theorem~\ref{main1}:]
Taking~$\varphi\equiv 1$ in~\eqref{GF3} we see that
\begin{equation*}\begin{split}
&\int_\Omega \Big(
\Ric (\nabla u,\nabla u)+|H_u|^2-\big|\nabla|\nabla u|\big|^2
\Big)\,dV\\
&\qquad\le
\int_{\partial\Omega} \left(\II(\tilde{\nabla} u,\tilde{\nabla} u)
-h'(u) |\tilde{\nabla} u|^2+h(u)\, f(u)+ (m-1)\big(h(u)\big)^2\,
H+h'(u)\big(h(u)\big)^2  \right) \, d\sigma.
\end{split}\end{equation*}
Hence, using~\eqref{cond0} and~\eqref{cond},
we obtain that
\begin{equation}\label{jfghur046} \int_\Omega \Big(
\Ric (\nabla u,\nabla u)+|H_u|^2-\big|\nabla|\nabla u|\big|^2
\Big)\,dV\le 0.\end{equation}
On the other hand, by Lemma~\ref{stimas},
$$ |H_u|^2-\big|\nabla|\nabla u|\big|^2\ge 0 \quad
{\mbox{ on }}\Omega,$$
and so~\eqref{jfghur046} gives
$$ \int_\Omega
\Ric (\nabla u,\nabla u)\,dV\le 0,
$$
which implies that
$$ \Ric (\nabla u,\nabla u)=0 \quad
{\mbox{ in }}\Omega.$$
Thus, the conclusion follows from Lemma~\ref{Fi1}.
\end{proof}

\end{document}